\documentclass[10pt]{article}
\usepackage{amssymb,amsmath,amsthm}

\usepackage{fullpage}

\newtheorem{theorem}{Theorem}
\newtheorem{lemma}[theorem]{Lemma}
\newtheorem{corollary}[theorem]{Corollary}
\newtheorem{proposition}[theorem]{Proposition}
\newtheorem{remark}[theorem]{Remark}

\title{The diachromatic number of digraphs
\thanks{Research supported by CONACyT-M{\' e}xico under Project 282280 and PAPIIT-M{\' e}xico under Project IN104915 and IN107218.}}
\author{Gabriela Araujo-Pardo\footnotemark[2] \and Juan Jos{\' e} Montellano-Ballesteros\footnotemark[2] \and Mika Olsen \footnotemark[3] \and Christian Rubio-Montiel\footnotemark[4]}

\begin{document}
\maketitle

\def\thefootnote{\fnsymbol{footnote}}
\footnotetext[2]{Instituto de Matem{\' a}ticas, Universidad Nacional Aut{\'o}noma de M{\' e}xico, Mexico City, Mexico. {\tt [garaujo|juancho]@math.unam.mx}.}
\footnotetext[3]{Departamento de Matem{\' a}ticas Aplicadas y Sistemas, UAM-Cuajimalpa, Mexico City, Mexico. {\tt
olsen.mika@gmail.com}.}
\footnotetext[4]{Divisi{\' o}n de Matem{\' a}ticas e Ingenier{\' i}a, FES Acatl{\' a}n, Universidad Nacional Aut{\'o}noma de M{\' e}xico, Naucalpan, Mexico. {\tt christian.rubio@apolo.acatlan.unam.mx}.}

\begin{abstract} 
We consider the extension to directed graphs of the concept of achromatic number in terms of acyclic vertex colorings. The achromatic number have been intensely studied since it was introduced by Harary, Hedetniemi and Prins in 1967. The dichromatic number is a generalization of the chromatic number for digraphs defined 
by Neumann-Lara in 1982
. A coloring of a digraph is an acyclic coloring if {each subdigraph induced by each
chromatic class is} acyclic, and a coloring is complete if for any pair of chromatic classes $x,y$, there is an arc from $x$ to $y$ and an arc from $y$ to $x$. The dichromatic and diachromatic numbers are, respectively, the smallest and the largest number of colors in a complete acyclic coloring. We give some general results for the diachromatic number and study it for tournaments.  
We also show that the interpolation property for complete acyclic colorings does hold and establish Nordhaus-Gaddum relations. 
\end{abstract}


\section{Introduction}
A complete coloring of an undirected graph $G$ is a vertex coloring of $G$ such that for every pair of colors there is at least one edge in G whose endpoints are colored with this pair of colors. The chromatic and achromatic numbers of $G$ are the smallest and the largest number of colors in a complete proper coloring of G, respectively.
The concept of achromatic number have been intensely studied in graphs since it was introduced by Harary, Hedetniemi and Prins \cite{MR0272662} in 1967, for more references of results related to this parameter see for instance \cite{MR3774452, MR2778722, MR3249588, AR16, CE97, MR0441778, Yeg01}. 
The achromatic number has been extended to digraphs with two different colorings one by Edwards \cite{MR2998438} and another by Sopena \cite{MR3202296}. Edwards considered colorings such that the underlying graph is proper colored and the completeness such that for each ordered pair $(c,c')$ of distinct colors, there is at least one arc $xy$ such that $x$ has color $c$ and $y$ has color $c'$, and he proved that a directed graph does not necessarily have a complete coloring and that determining whether one exists is an NP-complete problem. Sopena proposed another extension of the achromatic number, the oriented achromatic number, using the oriented chromatic number defining a completeness that corresponds to complete homomorphisms of oriented graphs. He proved that for every integers $a$ and $b$ with $2\le a\le b$, there exists an oriented graph $G_{a,b}$ with oriented chromatic number $a$ and oriented achromatic number $b$. He also studied the behavior of the oriented achromatic number adding or deleting a vertex or an arc. In both extensions of the achromatic number for digraphs, it is proven that the interpolation property does not hold.  
In this paper, we propose yet another extension of the achromatic number to digraphs en terms of acyclic colorings. We stress that with our extension, the interpolation property does hold.

A \emph{vertex coloring} of a digraph $D$ is called \emph{acyclic} if 
{each chromatic class  induces a subdigraph with } no directed cycles. %
The \emph{dichromatic number} $dc(D)$ of a digraph $D$ is the smallest $k$ such that $D$ admits an acyclic coloring, it was introduced by Neumann-Lara in \cite{MR693366} as a generalization of the chromatic number, for more references of results related to this parameter see for instance \cite{AO10, JGT3190150604, HOCHSTATTLER2017160, SIAM2017, LIN20112462, EJCLO17, VNeu00, MR3112565}. 
A coloring of a digraph $D$ is called \emph{complete} if for every {ordered}  pair $(i,j)$ of different colors there is at least one arc $(u,v)$ such that $u$ has color $i$ and $v$ has color $j$  \cite{MR2998438}. It is not hard to see that any acyclic coloring of $D$ with $dc(D)$ colors is a complete coloring. 
We define the \emph{diachromatic number} $dac(D)$ of a digraph $D$ as the largest number of colors for which there exists a complete and acyclic coloring of $D$. 
Hence, the dichromatic and diachromatic numbers of a digraph $D$ are, respectively, the smallest and the largest number of colors in a complete acyclic coloring of $D$.
The \emph{pseudoachromatic number} $\psi(D)$ of a digraph $D$ is the largest number $k$ for which there exists a complete coloring of $D$ using $k$ colors (see \cite{MR2998438,MR0256930}). 

Since the dichromatic number of a symmetric digraph is equal to the chromatic number of the underlying graph, the diachromatic number of a symmetric digraph is equal to the achromatic number of the underlying graph; and the pseudoachromatic number of a symmetric digraph is equal to the pseudoachromatic number of the underlying graph, the dichromatic number,  the diachromatic number and the pseudoachromatic number of a digraph generalizes the chromatic number, the achromatic number and the  pseudoachromatic number of a graph; recall that the \emph{underlying graph} $G_D$ of a digraph $D$ is obtained from $D$ changing all symmetric and asymmetric arcs by edges. Note that the chromatic number of the underlying graph is an upper bound for the dichromatic number of a digraph and the pseudoachromatic number of a digraph is a lower bound for the pseudoachromatic number of the underlying graph, but this kind of relation can not be established between the diachromatic and the achromatic numbers.

This paper is organized in six sections: the second section contains general results for the diachromatic number of a digraph, in the third section we study the diachromatic number of tournaments; in the four and five sections we generalize some results on graphs, given in Chapter 12  of \cite{MR2450569}, for digraphs; also in four section, we define the concept of dihomorphisms and show that the interpolation property does hold for complete acyclic colorings; and five section establishes the Nordhaus-Gaddum relations. Finally, the last section has some 
conclusions and future work. 


\section{Definitions and basic results}

In this paper, we consider finite digraphs.  The arc $uv\in A(D)$ is symmetric if $vu\in A(D)$ and asymmetric if $vu\notin A(D)$. A digraph is symmetric (resp. asymmetric) if every arc of $D$ is symmetric (resp. asymmetric); for general concepts see \cite{MR2107429}.
Let $k$ be a natural number. A \emph{vertex coloring} $\varsigma$ of a digraph $D$ with $k$ colors is a surjective function that assigns to each vertex of $D$ a color of $[k]:=\{1,\dots,k\}$. {For each $i\in[k]$, the set  $\varsigma^{-1}(i)\subseteq V(D)$ will be called {\it chromatic class}}.  The coloring $\varsigma$ of $D$ with $k$ colors is called \emph{acyclic} if {no chromatic class  induces a subdigraph with a  directed cycle}. {An acyclic coloring with $k$ colors is denoted (for short) as $k$-coloring}. 
The \emph{dichromatic number} $dc(D)$ of a digraph $D$ is the smallest $k$ such that $D$ admits an acyclic coloring \cite{MR693366}.
An acyclic coloring of a digraph $D$ with $k$ colors is called \emph{complete} if for every  {ordered} pair $(i,j)$ of different colors there is at least one arc $(u,v)$ such that $u$ has color $i$ and $v$ has color $j$  \cite{MR2998438}. It is not hard to see that any $dc(D)$-coloring of $D$ is a complete coloring. 
We define the \emph{diachromatic number} $dac(D)$ of a digraph $D$ as the largest number $k$ for which there exists a complete and acyclic $k$-coloring of $D$. 

From the definition, we obtain the following for any digraph $D$ of order $n$:

\begin{equation}\label{eq1}1\leq dc(D)\leq dac(D)\leq \psi(D) \leq n.\end{equation}

The \emph{converse} digraph $D^{op}$ of $D$ is obtained replacing each arc $(u,v)$ of $D$ by the arc $(v,u)$. The complement $D^{c}$ of a digraph $D$ is that digraph whose vertex set is $V(D)$ and where $(u,v)$ is an arc of $D$ if and only if $(u,v)$ is not an arc of $D$.
It is not difficult to prove the following remark.

\begin{remark}
If $G$ is a graph, and $\overrightarrow{G}$ is an orientation of $G$ we have that \[dc(\overrightarrow{G})=dc(\overrightarrow{G}^{op}), \\ dac(\overrightarrow{G})=dac(\overrightarrow{G}^{op}), \textrm{ and } \psi(\overrightarrow{G})=\psi(\overrightarrow{G}^{op}).\]
\end{remark}

The following theorem gives a rather simple (and sharp) bound for the diachromatic number of a digraph in terms of its size.

\begin{theorem}\label{teo11}
Let $D$ be a digraph of size $m$. Then $\psi(D)\leq\left\lfloor \frac{1+\sqrt{1+4m}}{2}\right\rfloor$. 
Moreover, if $\overrightarrow{M}$ is an oriented matching of size $m$, then $dac(\overrightarrow{M})=\left\lfloor \frac{1+\sqrt{1+4m}}{2}\right\rfloor$.
\end{theorem}
\begin{proof}
Since $m\geq 2\binom{\psi}{2}=\psi(D)(\psi(D)-1)$, we obtain $\psi(D)\leq \frac{1+\sqrt{1+4m}}{2}$ and the result follows.
Let $\overrightarrow{M}$ be an oriented matching of size $m$. We exhibit a $k$-complete and acyclic coloring of $\overrightarrow{M}$ for $m=k(k-1)$ and $k\geq 1$. 
Let $A(\overrightarrow{M})=\left\{(x_i,y_i)\colon i\in \left\{1,\dots ,n\right\}\right\}$ and 
$Y_j:=\left\{y_i\in V(\overrightarrow{M})\colon (j-1)(k-1)+1\leq i \leq j(k-1) \right\}$ for $1\leq j \leq k.$ Notice that the coloring $\phi_i\colon Y_j\longrightarrow \left\{1,2,\dots,j-1,j+1,\dots,k\right\}$ that assign different colors to different elements of $Y_j$ is bijective because $Y_j$ has exactly $k-1$ elements. 

We define $\phi\colon V(\overrightarrow{M})\longrightarrow \left\{1,\ldots ,k\right\}$ as $\phi(x_i)=l$ and $\phi(y_i)=\phi_l(y_i)$ for $(l-1)(k-1)+1\leq i \leq l(k-1)$ and $l\in \left\{1,\dots ,k\right\}$. The $k$-coloring is complete because for any pair of different colors $a$ and $b$ all the vertices $x_i$ for $(a-1)(k-1)+1\leq i \leq a(k-1)$ satisfies that  $\phi(x_i)=a$ and, by construction, always exits exactly one vertex in $Y_a$ of color $b$.  
Furthermore, if $m\geq k(k-1)$ and $m'=k(k-1)$, we can repeat some colors used previously on the vertices $\left\{x_i,y_i\right\}$ for $i\in \{m'+1,\dots ,m\}$ preserving the property of $\phi$.

Finally, it is not difficult to prove that if $k=\left\lfloor \frac{1+\sqrt{1+4m}}{2}\right\rfloor$ then $m\geq k(k-1)$ and we have a $\left\lfloor \frac{1+\sqrt{1+4m}}{2}\right\rfloor$-complete and acyclic coloring of $\overrightarrow{M}$. 
{By } the upper bound given in this theorem, we complete the proof. 
\end{proof}

The following theorems give some results related to the diachromatic number of digraphs in terms of their dichromatic number and structural properties. 

\begin{theorem}\label{upper3}
For every asymmetric digraph $D$ of order $n$, $dac(D)\leq\left\lceil \frac{n}{2}\right\rceil.$
\end{theorem}
\begin{proof}
Since every complete coloring of $D$ has at most one chromatic class of {cardinality} $1$, the result follows.
\end{proof}

Since the dichromatic number of a non acyclic digraph is at least $2$, we obtain the following corollary, which  is a generalization of a theorem given by Shaoji Xu \cite{MR1108075}. It establishes an upper bound for $dac(D)-dc(D)$ in terms of the order of $D$.

\begin{corollary}
For every non acyclic digraph $D$ of order $n$, \[dac(D)-dc(D)\leq \frac{n-3}{2}.\]
\end{corollary}

Let $D$ be a digraph of order $n$ whose $n$ vertices are listed in some specified order. In a \emph{greedy coloring} of $D$, the vertices are successively colored with positive integers according to an algorithm that assigns to the vertex under consideration the smallest available color. Hence, if the vertices of $D$ are listed in the order $v_1,v_2,\dots,v_n$, then the resulting greedy coloring $\varsigma$ assigns the color $1$ to $v_1$, that is, $\varsigma(v_1)=1$. If $v_1$ and $v_2$ are not a $2$-cycle, then assign $\varsigma(v_2)=1$, else $\varsigma(v_2)=2$. In general suppose that the first $j$ vertices $v_1,v_2,\dots,v_j$, where $1\leq j < n$, in the sequence have been colored with the colors $1,\dots ,{t-1}$. Let $\{C_i\}_{i=1}^{t-1}$ be the set of chromatic classes. 
Consider the vertex $v_{j+1}$, if there exists a chromatic class $C_i$ for which, either $N^+(v_{j+1})\cap C_i=\emptyset$ or $N^-(v_{j+1})\cap C_i=\emptyset$, then $\varsigma(v_{j+1})=i$, else $\varsigma(v_{j+1})=t$. When the algorithm ends, the vertices of $D$ have been assigned colors from the set $[k]$ for some positive integer $k$. Thus, 
\[dc(D)\leq k \leq dac(D)\]
 and so $k$ is an upper bound for the dichromatic number of $D$ and a lower bound for the diachromatic number of $D$.

It is useful to know how the diachromatic number of a digraph can be affected by the removal of a single vertex. 

\begin{theorem}\label{teo6}
For each vertex $u$ in a nontrivial digraph $D$, \[dac(D)-1\leq dac(D-u)\leq dac(D).\]
\end{theorem}
\begin{proof}
Let $dac(D)=l$, and let $\varsigma$ be a complete $l$-coloring of $D$ where the set of vertices colored $l$ is $U$ and suppose that $u\in U$. Therefore, the partial $l$-coloring of $\varsigma$ of $D$ restricted to $D-U$ is a complete $(l-1)$-coloring. We can obtain a complete coloring using a greedy coloration for the remaining vertices $x\in U-u$. Hence, $dac(D-u)\geq l-1=dac(D)-1$.

Let $dac(D-u)=k$, and consider a complete $k$-coloring of $D-u$. We can obtain a complete coloring of $D$ using a greedy coloration for $u$. Therefore, $dac(D)\geq k=dac(D-u)$.
\end{proof}

The following result is an immediate consequence of Theorem \ref{teo6}.

\begin{corollary} \label{cor1}
For every induced subdigraph $H$ of a digraph $D$, \[dac(H)\leq dac(D).\]
\end{corollary}

\begin{corollary}\label{coroteo2}
Every digraph $D$ with a vertex partition $(X,Y)$, such that for every $x\in X$ and $y\in Y$, $(x,y)\in A(D)$ has \[dac(D)\geq \min\{|X|,|Y|\}.\]
\end{corollary}
\begin{proof}
Let $X=\{x_1,\dots,x_r\}$ and $Y=\{y_1,\dots,y_s\}$. Suppose that $r\leq s$ and color the vertices $x_i$ and $y_i$ with $i$ if $i\in [r]$. The digraph $D'=D[X\cup \{y_i\}_{i\le r}]$ is an induced digraph of $D$ with diachromatic number $r$. By Corollary \ref{cor1}, $dac(D)\ge dac(D')=\min\{|X|,|Y|\}$. 
\end{proof}

By Theorem \ref{teo6}, the removal of a single vertex from a digraph $D$ can result in a digraph whose diachromatic number is either one less than or is the same as the diachromatic number of $D$; whereas there exist three possibilities when a single edge is removed.

\begin{theorem}\label{teo12}
For each arc $f=(u,v)$ in a nonempty digraph $D$, \[dac(D)-1\leq dac(D-f)\leq dac(D)+1.\]
\end{theorem}
\begin{proof}
Let $dac(D)=k$, then there exists a complete $k$-coloring of $D$, where the colors assigned to $u$ and $v$ are the same or distinct. If $u$ and $v$ have the same color assigned, the complete $k$-coloring of $D$ is also a complete $k$-coloring of $D-f$. Hence, $dac(D-f)\ge dac(D)>dac(D)-1$  then  $dac(D-f)\geq dac(D)-1$. Assume that $u$ and $v$ have different colors, say $u$ is colored $k-1$ and $v$ is colored $k$. If the resulting $k$-coloring of $D-f$ is not a complete $k$-coloring, then no directed cycle in $D-f$ is bicolored with the colors $k$ and $k-1$. Hence, every vertex colored $k$ may be recolored $k-1$, resulting in a complete $(k-1)$-coloring of $D-f$. In any case, $dac(D-f)\geq k-1 = dac(D)-1$.

Let $dac(D-f)=l$, then there exists a complete $l$-coloring of $D-f$. The vertices $u$ and $v$ are either assigned distinct colors or the same color. If $u$ and $v$ are assigned distinct colors, then the complete $l$-coloring of $D-f$ is also a complete $l$-coloring of $D$. Hence, we may assume that $u$ and $v$ are assigned the same color, say $l$. The complete $l$-coloring of $D-f$ is also a complete coloring of $D$ with $l$ colors. If the set of vertices colored $l$ is acyclic in $D$, then the complete $l$-coloring of $D-f$ is also a complete $l$-coloring of $D$. Assume that the set of vertices colored $l$ has monochromatic directed cycles in $D$, where $f$ is an arc of any monochromatic cycle. 
If there is a chromatic class $C_j$, for some $j\in[l-1]$,  such that $u\cup C_j$ (resp. $v\cup C_j$) induces an acyclic digraph, then by recoloring the vertex $u$ ($v$ resp.) with the color $j$, we obtain an acyclic  complete coloring using a greedy coloration  for the remaining vertices $x$ colored $l$. 
Assume that for each $i\in [l-1]$ the vertex set $C_i\cup u$ y $C_i\cup v$ contains a directed cycle, then recoloring $v$ with the color $l+1$. The  resulting coloring is an acyclic complete $(l+1)$-coloring of $D$.
In any case, $dac(D)\geq l-1=dac(D-f)-1$.
\end{proof}

A digraph $D$ is \emph{$k$-minimal} (with respect to diachromatic number) if $dac(D)=k$ and $dac(D-f)<k$ for every arc $f$ of $D$. By Theorem \ref{teo12}, if $D$ is a $k$-minimal digraph, then $dac(D-f)=dac(D)-1$. Since $k$-minimal digraphs have diachromatic number $k$, the size of every such graph is at least $k(k-1)$ and every chromatic class {induces an arc-less digraph}, else if $f$ is an arc of {an induced subdigraph of a chromatic class}, the complete $k$-coloring of $D$ is also a complete $k$-coloring of $D-f$. Hence, $dac(D-f)\ge dac(D)>dac(D)-1$, contradicting that $D$ is a $k$-minimal digraph. The following theorem is a generalization of a result given by Bhave \cite{MR532949} characterized graphs that are $k$-minimal in terms of their size.

\begin{theorem}
Let $D$ be a digraph with diachromatic number $k$. Then $D$ is $k$-minimal if and only if its size is $k(k-1)$.
\end{theorem}
\begin{proof}
Assume first that the size of $D$ is $k(k-1)$. Then for every arc $f$ of $D$, the size of $D-f$ is $k(k-1)-1$. Since the size of $D-f$ is less than $k(k-1)$, it follows that $dac(D-f)<k=dac(D)$ and that $D$ is $k$-minimal.
We now verify the converse. Assume, to the contrary, that there is a $k$-minimal graph $H$ whose size $m$ is not $k(k-1)$. Since $dac(H)=k$, it follows that $m\geq k(k-1)$. Since $m\not=k(k-1)$, it follows that $m\geq k(k-1)+1$. Let $\phi$ be a complete $k$-coloring of $H$. For every two distinct colors $i,j\in[k]$, there exist an arc such that its vertices are colored $i$ and $j$. Since there are only $\binom{k}{2}$ pairs of two distinct colors from the set $[k]$ and two arc for each pair, there are two distinct arc $f=(u,v)$ and $f'=(u',v')$ such that $\phi(u)=\phi(u')$ and  $\phi(v)=\phi(v')$, then $\phi$ is also a complete $k$-coloring of $H-f$, contradicting the assumption that $H$ is $k$-minimal.
\end{proof}


\section{Tournaments}

A \emph{tournament} $T$ of order $n$ is an orientation of the complete graph $K_n$. An acyclic  tournament $T$ is \emph{transitive}, the vertex set $V(T)=\{v_1,\dots,v_n\}$ of a transitive tournament has a unique (acyclic) order $(v_1,\dots,v_n)$, where $N^-(v_1)=\emptyset$ and $N^-(v_i)=\{v_1,\ldots v_{i-1}\}$ for $1<i\le n$. 

Let $T=(V, A)$ be a tournament. A transitive subtournament  $T' = (V',  A')$ of  $T$ is \emph{discordant} if for every $x \in V\setminus V'$  there is a pair $\{z, w\} \subseteq V'$ such that $\{xz, wx\}\subseteq   A$.
Let  $\Xi_2(T)$ be the minimum order of a discordant subtournament of $T$.

\begin{lemma}\label{1} Let $T $ be a tournament of order $n\geq 3$. Then
$$\Xi_2(T)\leq 2log_2\Big(\frac{2n+2}{3}\Big).$$ \end{lemma}

\begin{proof} For every $xy\in A$ let $C_3(xy)$   be the number of directed triangles of $T$ with $xy$ as
an arc, and $TT^*(xy)$ be the number of transitive triangles $C$ of $T$ such that $x$ is the source
and $y$ is the sink of $C$. Let $x_0y_0\in A$ such that $C_3(x_0y_0) + TT^*(x_0y_0)$ is maximum. 

\noindent {\bf Claim 1} $C_3(x_0y_0) + TT^*(x_0y_0) \geq \frac{n-2}{3}$. 

\noindent   Let $C_3(T)$ and $TT_3(T)$ be the number of directed triangles and transitive triangles in $T$, respectively. Observe that 
$$\sum\limits_{xy\in A}\left( C_3(xy) + TT^*(xy)\right) = 3C_3(T) + TT_3(T) = {{n}\choose{3}} + 2C_3(T),$$ 
by an average argument there exists $zw\in A$ such that $$C_3(zw) + TT^*(zw) \geq \frac{{{n}\choose{3}} + 2C_3(T)}{{{n}\choose{2}}}\geq \frac{n-2}{3} + \frac{2C_3(T)}{{{n}\choose{2}}}$$ and the claim follows.

Let $B^+ =\{ z\in V : \{x_0z, y_0z\} \subseteq A\}$ and $B^- =\{ z\in V : \{zx_0, zy_0\} \subseteq A\}$. By definition, it follows that $|B^+| + |B^-| = n-2 -C_3(x_0y_0) - TT^*(x_0y_0)$.

Let $B_0^+= B^+$ and let $z_0\in B^+_0$ be a vertex with maximum in-degree in $T[B^+_0]$. For each $i\geq 1$ let $B_i^+ = B_{i-1}^+ \setminus N^-[z_{i-1}]$ and $z_i\in B_i^+$ be a vertex with maximum in-degree in  $T[B^+_i]$. Let $j$ be the minimum integer such that $B_{j+1}^+= \emptyset$.  Observe that $j\leq log_2(|B^+|+1)-1$ and that $V' =\{z_j, z_{j-1}, \dots, z_0, x_0, y_0\}$ induces a transitive subtournament of $T$ such that for every $w\in V\setminus (B^-\cup V')$ there is a pair $\{x, y\}\subseteq V'$ such that $\{xw, wy\}\subseteq A$.

Let  $B^* = \{w\in B^- : \{wz_j, wz_{j-1}, \dots, wz_0, wx_0, wy_0\} \subseteq A\}$,  $B_0^-= B^*$ and $w_0\in B^-_0$ be a vertex with maximum out-degree in $T[B^-_0]$. As before, for each $i\geq 1$ let $B_i^- = B_{i-1}^- \setminus N^+[w_{i-1}]$ and $w_i\in B_i^-$ be a vertex with maximum out-degree in  $T[B^-_i]$. Let $q$ be the minimum integer such that $B_{q+1}^-= \emptyset$.  Observe that $q\leq log_2(|B^*|+1)-1 \leq log_2(|B^-|+1)-1$ and that $V' =\{z_j, z_{j-1}, \dots, z_0, x_0, y_0, w_0, w_1,\dots, w_q\}$ induces a discordant  subtournament of $T$  of order 
$$
\begin{array}{lcl}
2+ log_2(|B^+|+1) + log_2(|B^-|+1)   &  = & 2 + log_2\left((|B^+|+1)(|B^-|+1)\right)  \\
  &  = & log_2\left(4(|B^+|+1)(|B^-|+1)\right)  .
  \end{array}
$$
Since $|B^+| + |B^-|= n-2 -C_3(x_0y_0) - TT^*(x_0y_0)$  we see that  
$$
\begin{array}{lcl}
log_2\big(4(|B^+|+1)(|B^-|+1)\big)  &  \leq & log_2\left( 4\left(\dfrac{n-C_3(x_0y_0) - TT^*(x_0y_0)}{2}\right)^2\right)   \\
  &  = & 2log_2\left(n -C_3(x_0y_0) - TT^*(x_0y_0)\right). 
  \end{array}
$$

From here, and by Claim 1, the result follows.\end{proof}

\begin{theorem}\label{Torneo}  Let $T$ be a tournament of order $n\geq 3$.   Hence, $$\frac{n}{2log_2(\frac{2n+2}{3})} \leq dac(T)\leq \left\lceil\frac{n}{2}\right\rceil.$$
\end{theorem}
\begin{proof}  Given any complete vertex coloring $\Gamma$ of $T$, it follows that there is no pair of colors $i,j$ such that $|\Gamma^{-1}(i)| = |\Gamma^{-1}(j)| =1$, otherwise either there is no $\Gamma(i)\Gamma(j)$-arc or there is no $\Gamma(j)\Gamma(i)$-arc. Thus, except for at most one color, every color appear at least twice in $V$. Therefore, $dac(T)\leq \lceil\frac{n}{2}\rceil$. For the lower bound, let $\{V_1, \dots, V_k\}$ be a partition of $V$ such that, for every $i\in[ k-1]$, $T[V_i]$ is a discordant subtournament of $T[ V\setminus \bigcup\limits_{j\in[i-1]} V_j]$, and $k$ is maximum. Let $\Gamma: V\rightarrow [k]$ be the vertex coloring of $T$ such that  $\Gamma(x) = j$ if and only if $x\in V_j$.  On the one hand, for every pair $i,j\in[k]$, if $i<j$, since $V_i$ is a discordant subtournament of $T[ V\setminus \bigcup\limits_{j\in[i-1]} V_j]$, it follows that in $T$  there are $V_iV_j$-arcs and $V_jV_i$-arcs, and therefore $\Gamma$ is a complete coloring of $T$.  On the other hand, 
by Lemma \ref{1},  there is a partition  $\{V'_1, \dots, V'_{k'}\}$  of $V$  such that  for every $i\in[ k'-1]$, $T[V'_i]$ is a discordant subtournament of $T[ V\setminus \bigcup\limits_{j=1}^{i-1} V'_j]$, with 
$|V'_i|\leq  2log_2\left(\frac{2\left(n- \sum\limits_{j\in[i-1]} |V'_j|\right)+2}{3}\right) \leq 2log_2\left(\frac{2n+2}{3}\right)$, 
and $|V'_{k'}|\leq 2$.  Since $2 \leq 2log_2\left(\frac{2n+2}{3}\right)$,  
$$dac(T)\geq k\geq k' \geq \frac{n}{2log_2(\frac{2n+2}{3})}$$ 
and the result follows. \end{proof}

Let $\mathbb{Z}_{2m+1}$ be the cyclic group of integers modulo $2m+1$ $(m\geq 1)$ and $J$ a nonempty subset of $\mathbb{Z}_{2m+1}\setminus \{0\}$ such that $\left\vert \{-j,j\}\cap J\right\vert =1$ for every $j\in J$. The \textit{circulant tournament} $\overrightarrow{C}_{2m+1}(J)$ has vertex-set $V(\overrightarrow{C}_{2m+1}(J))=\mathbb{Z}_{2m+1}$ and arc-set $A(\overrightarrow{C}_{2m+1}(J))=\left\{ (i,j):i,j\in \mathbb{Z}_{2m+1}\text{ and }j-i\in J\right\} $.

\begin{corollary}\label{trans}
Let $T$ be a circulant tournament or a transitive tournament of order $n$. Then $$dac(T)=\psi(T)=
\left\lceil {n}/{2}\right\rceil.$$
\end{corollary}
\begin{proof}
Let $\overrightarrow{C}_{n}(J)$ be a circulant tournament, with $n=2m+1$, the coloring $\varsigma(0)=0$ and $\varsigma(i)=\varsigma({2m+1-i})=i$ for $i\in [m]$ defines a complete $m+1$-coloring of $\overrightarrow{C}_{2m+1}(J)$ and $m+1=\left\lceil {2m+1}/{2}\right\rceil=\left\lceil {n}/{2}\right\rceil$. For a  transitive tournament $T$ of order $n$ with the acyclic order of its vertex set  $(v_1,v_2,\dots,v_{n})$, the coloring $\varsigma(v_i)=\varsigma(v_{n+1-i})=i$ for $i\in \{1,\dots,\left\lceil \frac{n}{2}\right\rceil\}$ defines a complete $\left\lceil \frac{n}{2}\right\rceil$-coloring of $T$. The upper bound in Theorem \ref{Torneo} completes the proof. 
\end{proof}


\begin{corollary}
If $T$ is a tournament of order $n$, then ${n}/{2}\leq dc(T)dac(T)$.
\end{corollary}
\begin{proof}
Let $\varsigma$ be a $dc(T)$-coloring and $x$ the order of the largest chromatic class $i$, therefore, $n\leq xdc(T)$. By Corollaries \ref{cor1} and \ref{trans}, we obtain $\frac{x}{2}\leq dac(T)$ because the subdigraph induced by $\varsigma^{-1}(i)$ is a transitive tournament. 
\end{proof}

Recall that a digraph is \emph{strongly connected} if for every pair of vertices $u$ and $v$ there exist a directed $u-v$ walk.  The \emph{strongly components} of a digraph $D$ form a partition $\sim$ into subdigraphs that are themselves strongly connected. Such partition induces the digraph $\tilde{D}:=D/\sim$.
\begin{theorem}\label{teo strong comp}\cite{MR2107429}
If $T$ is a tournament with (exactly) $k$ strong components, then $\tilde{T}$ is the transitive tournament of order $k$.
\end{theorem}

\begin{corollary}
If $T$ is a  tournament   with (exactly) $k$ strong components, then $\left\lceil {k}/{2}\right\rceil\leq dac(T)$. 
\end{corollary}

\begin{proposition}\label{teo3}
Every digraph $D$ that admits two vertex disjoint transitive tournaments $T_r$ of order $r$ and $T_s$ of order $s$ has $dac(D)\geq \min\{r,s\}+\lfloor\frac{s-r}{2}\rfloor$.
\end{proposition}
\begin{proof}
Let $V(T_r)=\{x_1,\dots,x_r\}$ and $V(T_s)=\{y_1,\dots,y_s\}$. Suppose that $r\leq s$ 
and let $k=\lfloor\frac{s-r}{2}\rfloor$. 
Color the vertices $x_i$ and $y_{k+r+1-i}$ with $i$ if $i\in\{1,\dots,r\}$, color the vertices $\{y_{k+i},y_{k+r+1-i}\}$ with $i$ if $i\in\{r+1,\dots,r+k\}$ 
and color the remaining vertices using a greedy coloration. The coloring is complete since for the pair of distinct colors $(i,j)$, with $i,j\le r$,  there are at least an arc $(x_i,x_j)$ or $(y_i,y_j)$ such that $x_i$ or $y_i$ has the color $i$ and $x_j$ or $y_j$ has the color $j$; 
and for each integer $j$, with $r<j\le r+k$ and $i<j$ there are at least the arcs $(y_{i'},y_{k+j})$ and $(y_{k+r+1-j},y_{i'})$ such that $y_{i'}$ has the color $i$ and the vertices $y_{k+j}$ and $y_{k+r+1-j}$ have the color $j$.
\end{proof}

\begin{theorem}
If $T$ is a tournament of order $n$, then $\sqrt{n}-1/2\leq dac(T)$.
\end{theorem}
\begin{proof}
Let $\varsigma$ be a $dac(T)$-coloring and let $x$ be the order of the largest chromatic class $i$. Since the subdigraph induced by $\varsigma^{-1}(i)$ is a transitive tournament, by Corollary  \ref{trans} we obtain $\frac{x}{2}\leq dac(T)$, and $-x\geq -2dac(T)$.
By Proposition \ref{teo3}, $dac(T)\geq y$ where $y$ is the order of the second largest chromatic class $j$. Since $y\geq \frac{n-x}{dac(T)-1}$, we obtain \[dac(T)(dac(T)-1)\geq n-x\geq n-2dac(T).\]
Finally, we solve $dac(T)^{2}+dac(T)-n\geq0$ obtaining $dac(T)\geq \frac{\sqrt{1+4n}-1}{2}$ and the result follows.
\end{proof}

Let $F$ be a digraph. The digraph $D$ is \emph{$F$-free} if $D$ has no subdigraph isomorphic to $F$.  
A tournament $F$ is a \emph{hero} if and only if there exists $c > 0$ such that every $F$-free tournament $D$ has a transitive subset of cardinality at least $c|V(D)|$, see \cite{MR2995716}. Let us say a tournament $F$ is a \emph{hero}, if there exists $c'$ (depending on $F$) such that every $F$-free tournament has dichromatic number at most $c'$.
\begin{corollary}
If $T$ is a hero then $dac(T)\in \theta(n)$.
\end{corollary}

In the following sections, we generalize some results on graphs, given in Chapter 12  of \cite{MR2450569}, for digraphs; in Section \ref{section4}, we define the concept of dihomorphisms and show that the interpolation property does hold for complete acyclic colorings; and in Section \ref{section5}, we establish the Nordhaus-Gaddum relations. 


\section{The interpolation theorem}\label{section4}
Recall that two vertices are adjacent if they are the vertices of a $2$-cycle. An \emph{elementary dihomomorphism} of a digraph $D$ is obtained by identifying two nonadjacent vertices $u$ and $v$ of $D$. The vertex obtained by identifying $u$ and $v$ may be denoted by either $u$ or $v$. Thus the resulting dihomomorphic image $D'$ can be considered to have vertex set $V(D)\setminus \{u\}$ and arc set \[
\begin{array}{lll}
A(D') & = & \{(x,y):(x,y)\in A(D),x,y\in V(D)\setminus\{u\}\} \\[1ex]
&  &\cup ~ \{(v,x):(u,x)\in A(D),x\in V(D)\setminus\{u,v\}\}\\[1ex]
&  &\cup ~ \{(x,v):(x,u)\in A(D),x\in V(D)\setminus\{u,v\}\}.\\[1ex]
\end{array}
\]
Alternatively, the mapping $\epsilon \colon V(D)\rightarrow V(D')$ defined by \[
\epsilon(x) =
\begin{cases}
x & \text{if  } x\in V(D)\setminus\{u,v\};\\
v & \text{if } x\in \{u,v\};
\end{cases}
\]
is an elementary dihomomorphism from $D$ to $D'$. The dihomomorphic image $\epsilon(D)$ of a digraph $D$ obtained from an elementary dihomomorphism $\epsilon$ is also referred to as an elementary dihomomorphic image. Not only is $D'$ a dihomomorphic image of $D$, a digraph $F$ is a dihomomorphic image of a digraph $D$ if and only if $F$ can be obtained by a sequence of elementary dihomomorphisms beginning with $D$.

The fact that each dihomomorphic image of a digraph $D$ can be obtained from $D$ by a sequence of elementary dihomomorphisms tells us that we can obtain each dihomomorphic image of $D$ by an appropriate partition \[P=\{V_1,V_2,\dots,V_k\}\] of $V(D)$ into acyclic sets such that $V(F)=\{v_1,v_2,\dots,v_k\}$, where $v_i$ is adjacent to $v_j$ if and only if some vertices $u$ and $u'$ in $V_i$ and some vertices $v$ and $v'$ in $V_j$ are arcs $(u,v)$ and $(v',u)$ of $D$. The partition $P$ of $V(D)$ then corresponds to the coloring $\varsigma$ of $D$ in which each vertex in $V_i$ is assigned the color $i$ $(1\leq i \leq k)$. In particular, if the coloring $\varsigma$ is a complete $k$-coloring, then $F$ is the complete symmetric digraph of order $k$.

Therefore, if a digraph $F$ is a dihomomorphic image of a digraph $D$, then there is a dihomomorphism $\phi$ from $D$ to $F$ and for each vertex $v$ in $F$, the set $\phi^{-1}(v)$ of those vertices of $D$ having $v$ as their image is acyclic in $D$. Consequently, each coloring of $F$ gives rise to a coloring of $D$ by assigning to each vertex of $D$ in $\phi^{-1}(v)$ the color that is assigned to $v$ in $F$. For this reason, the digraph $D$ is said to be $F$-colorable. This provide us the following remark: 

\begin{remark}\label{F}
If $F$ is a dihomomorphic image of a digraph $D$, then \[dc(D)\leq dc(F).\]
\end{remark}

\begin{theorem}\label{teo1}
Let $\epsilon$ be an elementary dihomomorphism of a digraph $D$, then \[dc(D)\leq dc(\epsilon(D))\leq dc(D)+1.\]
Moreover, $dc(\epsilon(D))=dc(D)$ if and only if there exists a $dc(D)$-coloring of $D$ in which the identified vertex share the color.
\end{theorem}
\begin{proof}
Suppose that $\epsilon$ identifies the nonadjacent vertices $u$ and $v$ of $D$. We have already noted the inequality $dc(D)\leq dc(\epsilon(D))$. Let $dc(D) = k$ and consider a $k$-coloring $\varsigma$ of $D$. Define a coloring $\varsigma'$ of $\epsilon(D)$ by \[
\varsigma'(x) =
\begin{cases}
\varsigma(x) & \text{if  } x\in V(D)\setminus\{u,v\};\\
k+1 & \text{if } x\in \{u,v\}.
\end{cases}
\]
Since $\varsigma'$ is a $(k+1)$-coloring of $\epsilon(D)$, 
\[dc(\epsilon(D))\leq k+1=dc(D)+1.\]

Clearly, if there exists a $dc(D)$-coloring of $D$ in which $u$ and $v$ are assigned the same color, we have that: 
\[\varsigma'(x) =
\begin{cases}
\varsigma(x) & \text{if  } x\in V(D)\setminus\{u,v\};\\
\varsigma(u)=\varsigma(v) & \text{if } x\in \{u,v\};
\end{cases}
\]
Conversely, if $dc(\epsilon(D))=dc(D)=k$. Let $\varsigma'$ be a $k$-coloring of $\epsilon(D)$, then the coloring of $D$ that assign to $\epsilon^{-1}(x)$ the color of $x$ is a coloring with the required property.  
\end{proof}

Beginning with a noncomplete digraph $D$, we can always perform a sequence of elementary dihomomorphisms until arriving at some complete graph. As we saw, a complete graph $K_k$ obtained in this manner corresponds to a complete $k$-coloring of $D$. Consequently, we have the following.

\begin{corollary}
The largest order of a complete graph that is a dihomomorphic image of a digraph $D$ is the diachromatic number of $D$.
\end{corollary}

The following theorem is a generalization of \emph{The Homomorphism Interpolation Theorem} due to Harary, Hedetniemi, and Prins \cite{MR0272662} and is an immediate consequence of Theorem \ref{teo1} (see also \cite{MR2450569}).

\begin{theorem}
Let $D$ be a digraph. For every integer $l$ with $dc(D)\leq l \leq dac(D)$ there is a dihomomorphic image $F$ of $D$ with $dc(F) = l$.
\end{theorem}
\begin{proof}
The theorem is certainly true if $l=dc(D)$ or $l=dac(D)$. Hence, we may assume that $dc(D)< l < dac(D)$. Suppose that $dac(D)=k$. Then there is a sequence \[D = D_0,D_1,\dots,D_t=K_k\] of digraphs where $D_i=\epsilon_i(D_{i-1})$ for some elementary dihomomorphism $\epsilon_i$ of $D_{i-1}$ for $i\in\{1,\dots,t\}$. Since $dc(D_0)< l < dc(D_t)=k$, there exists a largest integer $j$ with $j\in\{0,\dots,t-1\}$ such that $dc(D_j)<l$. Hence, $dc(D_{j+1})\geq l$. By Theorem \ref{teo1}, 
\[dc(D_{j+1})\leq dc(D_j)+1<l+1.\]
Hence,  $dc(D_{j+1})=l$.
\end{proof}

The Dihomomorphism Interpolation Theorem can be rephrased in terms of complete colorings, namely:

\emph{For a digraph $D$ and an integer $l$, there exists a complete $l$-coloring of $D$ if and only if $dc(D)\leq l\leq dac(D)$.}

We bound the diachromatic number of an elementary dihomomorphism of a digraph $D$ in terms of $dac(D)$.

\begin{theorem}\label{teo7}
If $\epsilon$ is an elementary dihomomorphism of a digraph $D$, then \[dac(D)-2\leq dac(\epsilon(D))\leq dac(D).\]
Moreover, for every noncomplete digraph $D$, there is an elementary dihomomorphism $\epsilon$ of $D$ such that $dac(\epsilon(D))=dac(D)$.
\end{theorem}
\begin{proof}
Let $\epsilon$ {be} an elementary dihomomorphism of a digraph $D$ that identifies the two nonadjacent vertices $u$ and $v$ and let the vertex in $\epsilon(D)$ obtained by identifying $u$ and $v$ be denoted by $v$. Let $dac(\epsilon(D))=k$ and consider a complete $k$-coloring $\varsigma$ of $\epsilon(D)$. Assigning the vertices $u$ and $v$ in $D$ the color $\varsigma(v)$ in $\epsilon(D)$ produces a complete $k$-coloring of $D$, and $dac(D)\geq k=dac(\epsilon(D))$.

Let $dac(D)=l$. Then $dac(D-u-v)\geq l-2$ by Theorem \ref{teo6}. Furthermore, $D-u-v$ is an induced subgraph of $\epsilon(D)$. By Corollary \ref{cor1},  $dac(\epsilon(D))\geq dac(D-u-v)\geq l-2$. Hence, $dac(\epsilon(D))\geq dac(D)-2$.

For every noncomplete digraph $D$, there is an elementary dihomomorphism $\epsilon$ of $D$ such that $dac(\epsilon(D))=dac(D)$.

Suppose that $dac(D)=k$. Hence, there exists a sequence \[D=D_0,D_1,\dots,D_t=K_k\] of digraphs, where $\epsilon_i(D_{i-1})=D_i$ for an elementary dihomomorphism $\epsilon_i$ of $D_{i-1}$ $(1\leq i\leq t)$. Thus, \[k=dac(D_t)\leq dac(D_{t-1}) \leq \dots \leq dac(D_1) \leq dac(D)=k.\]
Therefore, $dac(D_i)=k$ for all $i\in[t]$ and the result follows. 
\end{proof}

\section{Nordhaus-Gaddum relations}\label{section5}

The Nordhaus-Gaddum Theorems \cite{MR0256930,MR0078685} state \[\chi(G)+\chi(G^{c})\leq\alpha(G)+\chi(G^{c})\leq n+1\]
\[\chi(G)\chi(G^{c})\leq\alpha(G)\chi(G^{c})\leq \left(\frac{n+1}{2}\right)^{2}\]
for every graph $G$ of order $n$. For digraphs, we have the following results:

We have seen that for every digraph $D$ of order $n$ (see Equation \ref{eq1}), \[dc(D)\leq dac(D)\leq n.\] With the aid of Theorem \ref{teo1}, we show that $dac(D)$ can never be closer to $n$ than to $dc(D)$.

\begin{theorem}\label{teo9}
For every digraph $D$ of order $n$, \[dac(D)\leq \frac{dc(D)+n}{2}.\]
\end{theorem}
\begin{proof}
Let $dac(D)=k$. Then there is a sequence $D=D_0,D_1,\dots,D_t=K_k$ of graphs where $D_i=\epsilon_i(D_{i-1})$ for $1\leq i\leq t=n-k$ (since $t+k=n$) and an elementary dihomomorphism $\epsilon_i$ of $D_{i-1}$. By Theorem \ref{teo1}, $dc(\epsilon_i(D_{i-1}))\leq dc(D_{i-1}) + 1$ and so $dc(D_i)\leq dc(D_{i-1})+1$ for $1\leq i\leq t$. Therefore, \[\overset{t}{\underset{i=1}{\sum}}dc(D_{i})\leq\overset{t}{\underset{i=1}{\sum}}\left(dc(D_{i-1})+1\right)\]
and so $k\leq dc(D) + t = dc(D)+(n-k)$. Hence, $2k=2dac(D)\leq dc(D)+n$ and the result follows.
\end{proof}

The following result is the analogue of Theorem \ref{teo1} for complementary graphs.

\begin{theorem}\label{teo5}
If $\epsilon$ is an elementary dihomomorphism of a digraph $D$, then\[dc(D^{c})-1\leq dc(\epsilon(D)^{c})\leq dc(D^{c})+1.\]
\end{theorem}
\begin{proof}
We first show that $dc(D^{c})-1\leq dc(\epsilon(D)^{c})$. Let $dc((\epsilon(D))^{c})=k$ for some elementary dihomomorphism $\epsilon$ of $D$ that identifies two nonadjacent vertices $u$ and $v$ in $D$, where the vertex in $\epsilon(D)$ obtained by identifying $u$ and $v$ is denoted by $v$. Let there be given an $k$-coloring of $\epsilon(D)^{c}$. We may assume that the vertex $v$ in $\epsilon(D)$ is assigned the color $k$. Remember that, if $v$ and a vertex $w$ in $\epsilon(D)^{c}$ lie on a (directed) common cycle, $w$ has a different color to $k$. Assign to each vertex in $D^{c}$ distinct from $u$ the same color assigned to that vertex in $\epsilon(D)^{c}$ and assign $u$ the color $k+1$. Since no vertex in $D^{c}$ on a common cycle with $v$ is assigned the color $k$, this produces an $(k+1)$-coloring of $D$ and so \[dc(D^{c})\leq k+1=dc(\epsilon(D)^{c})+1.\]
Therefore, $dc(D^{c})-1\leq dc(\epsilon(D)^{c})$. 

Next, we show that $dc(\epsilon(D)^{c})\leq dc(D^{c})+1$. Let $dc(D^c) = k$ and consider a $k$-coloring $\varsigma$ of $D^c$. Define a coloring $\varsigma'$ of $\epsilon(D)^c$ by \[
\varsigma'(x) =
\begin{cases}
\varsigma(x) & \text{if  } x\in V(D)\setminus\{u,v\};\\
k+1 & \text{if } x\in \{u,v\}.
\end{cases}
\]
Since $\varsigma'$ is a $(k+1)$-coloring of $\epsilon(D)^c$, 
\[dc(\epsilon(D)^c)\leq k+1=dc(D^c)+1.\]
\end{proof}

\begin{theorem}(\emph{Nordhaus-Gaddum relations})\label{teo8}
If $D$ is a digraph of order $n$, then \begin{equation}\label{eq2} dc(D)+dc(D^{c})\leq \left\lceil \frac{4n}{3}\right\rceil \text{ and } dc(D)dc(D^{c})\leq \left(\frac{2n+1}{3}\right)^{2},\end{equation}
\begin{equation}\label{eq3} dac(D)+dc(D^{c})\leq \left\lceil \frac{3n}{2}\right\rceil \text{ and } dac(D)dc(D^{c})\leq\left(\frac{3n+1}{4}\right)^{2}. \end{equation}
\end{theorem}
\begin{proof}
Let $dac(D)=k$. Then there exists a sequence \[D = D_0,D_1,\dots,D_t=K_k\] of graphs, where $D_i=\epsilon_i(D_{i-1})$ for an elementary dihomomorphism $\epsilon_i$ of $D_{i-1}$ ($i\in[t]$). Then $t=n-k=n$. By Theorem \ref{teo5}, $dc(D_{i-1}^c)\leq dc(\epsilon_i(D_{i-1})^c)+1\leq dc(D_i^c)+2$. Thus, \[\overset{t}{\underset{i=1}{\sum}}dc(D_{i-1}^{c})\leq\overset{t}{\underset{i=1}{\sum}}\left(dc(D_{i}^{c})+2\right)\] and so $dc(D^{c})\leq dc(D_t^{c})+2t=dc(K_k^{c})+2(n-k)$. Hence, $2k+dc(D^{c})=2dac(D)+dc(D^{c})\leq 2n+1$. 

On one hand, since $dc(D)\leq dac(D)$, it follows that $dc(D)+dc(D^c)\leq \left\lfloor \frac{4n+2}{3} \right\rfloor = \left\lceil \frac{4n}{3}\right\rceil$ and the geometric mean of two positive real numbers never exceeds their arithmetic mean $dc(D)dc(D^{c})\leq \left(\frac{2n+1}{3}\right)^{2}$.

On the other hand, since $2dc(D^c)\leq dc(D^c)+n$, it follows that $dac(D)+dac(D^c)\leq \left\lfloor \frac{3n+1}{2} \right\rfloor = \left\lceil \frac{3n}{2}\right\rceil$ and then $dc(D)dc(D^{c})\leq \left(\frac{3n+1}{4}\right)^{2}$, completing the proof.
\end{proof}

\section{Conclusions and Future work}

As we state in the introduction, the diachromatic number generalizes the achromatic number and as it turned out, several classic results for the achromatic number are extended to results for the diachromatic number. 
To future work we propose extend other known results to the diachromatic number as well as to study proper results for digraphs, and study the diachromatic number in fixed families of digraphs or oriented graphs. For instance, inspired in the definition of the dichromatic number of a graph, the diachromatic number of a graph can be defined as the minimum diachromatic number over all possible orientations of the graph. 
On the other hand, we also propose study the \emph{diGrundy number}, defined as the the maximum positive integer $k$ for which a digraph $D$ has a greedy $k$-coloring, extending the results of undirected graphs to directed and oriented graphs and find this parameter for some families of digraphs. Also, we wish to improve the Nordhaus-Gaddum bounds or to find an infinite family of graphs  to establish the sharpness of such bound.


\bibliographystyle{plain}
\bibliography{biblio}

\end{document}